\documentclass[11pt,reqno]{amsart}
\usepackage{xifthen}
\usepackage{subfig}
\usepackage{pkgfile}
\newcommand{\p}{{\mathbb{P}_{n,\beta}}}

\begin{document}

\title[Phase transition in the two star Exponential Random Graph Model]{Phase transition in the two star Exponential Random Graph Model}

\author[S.\ Mukherjee]{Sumit Mukherjee$^\dagger$}  

\address{$\dagger$ Department of Statistics, Stanford University
\newline\indent Sequoia Hall, 390 Serra Mall, Stanford, California 94305}

\date{\today}





\maketitle

\begin{abstract}

This paper gives a way to simulate from  the  two star probability distribution on the space of simple graphs via auxiliary variables. Using this simulation scheme, the model is explored for various domains of the parameter values, and the phase transition boundaries are identified, and shown to be  similar as that of the Curie-Weiss model of statistical physics.  Concentration results are obtained for all the degrees, which further validate the phase transition predictions. 
\end{abstract}
\section{ Introduction}

A great number of models are used to do statistical analysis on network and graph data. This paper focuses on a simple model which goes beyond the Erdos Renyi model, namely the two star model studied in  \cite{PN}. By a two star is meant the following simple graph on $3$ vertices :
\begin{figure}[htbp]
\begin{center}
\includegraphics[height=2in,width=2in]{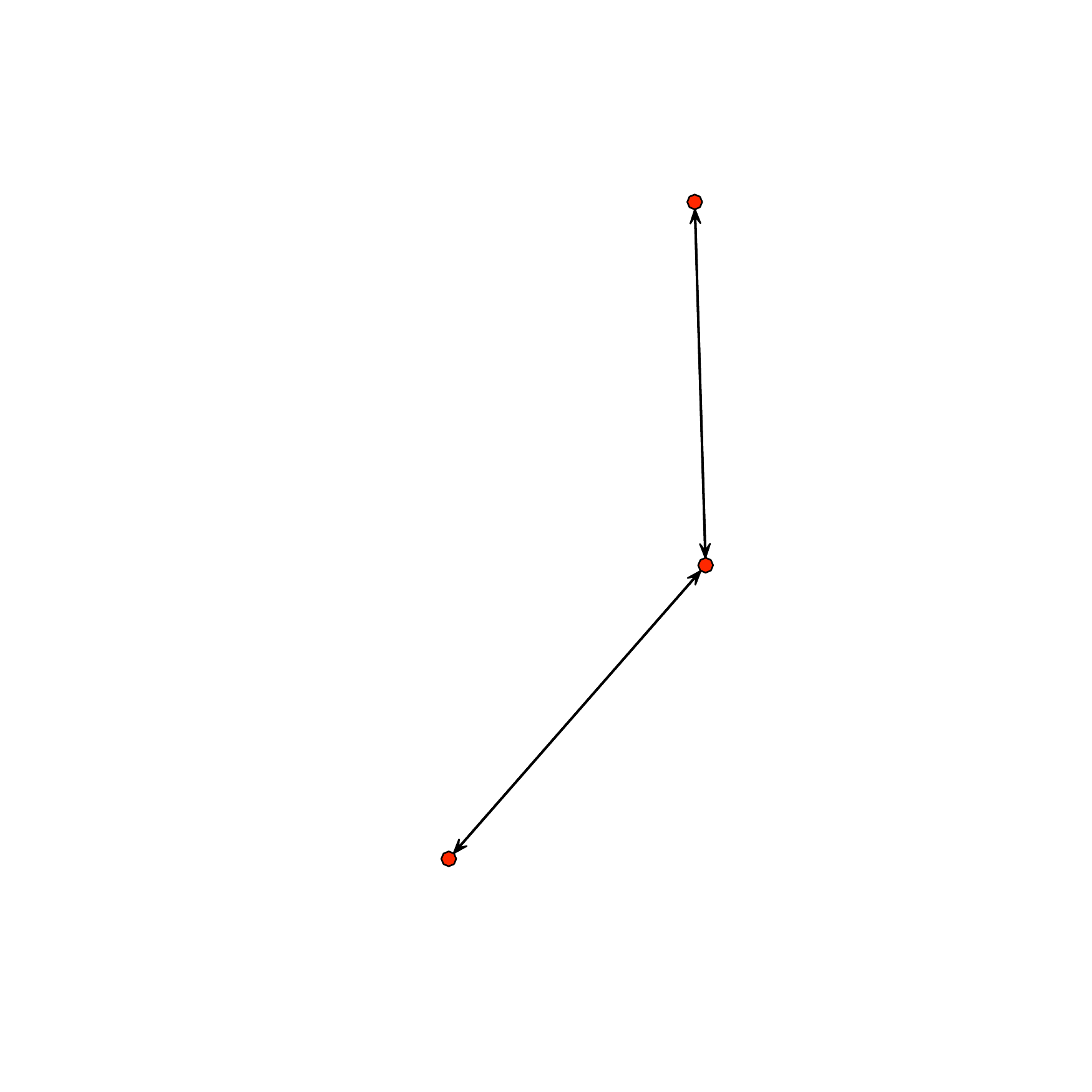} 
 \end{center}
\end{figure}

The two star model is the simplest of a wide class of models known as exponential random graph models(ERGM). This class of models were first studied by Holland and Leinhardt in \cite{HL}, and later developed by Strauss  in \cite{Strauss} and \cite{FS}. ERGM's are frequently used for modeling network data, the most common application area being social networks. For examples of such applications see \cite{ACW}, \cite{Newman}, \cite{PW}, \cite{RPKL},\cite{Snijders}, \cite{WF} and the references therein. ERGM produces a set of natural probability distributions on graphs, where the user can specify his/her choice of sufficient statistics  for the model. Below is given the formal definition of an ERGM. 
\\

\subsection{ Definition of ERGM}
For $n\in\mathbb{N}$ be a positive integer, let $\mathcal{X}_n$ denote the space of all simple graphs with vertices labeled $[n]:=\{1,2,\cdots,n\}$ .  Since a simple graph is uniquely identified by its adjacency matrix, a graph can be identified with its adjacency matrix, w.l.o.g.   $\mathcal {X}_n$ can be taken to be  the set of all symmetric $n\times n$ matrices ,with $0$ on the diagonal elements and $ \{0/1\}$ on the off-diagonal elements.  For $1\leq i\leq k$, let $T_i: {\mathcal X}_n\rightarrow \mathbb{R}$ be real valued statistics on the space of graphs. An ERGM with sufficient statistics $\{T_i,1\leq i\leq l\}$ is a probability distribution on $\mathcal{X}_n$ with probability mass function $$\frac{1}{Z_n(\beta)}\text{ exp}\{\sum_{i=1}^l\beta_iT_i(x)\},$$ where  $x=((x_{ij}))_{n\times n}\in \mathcal{X}_n,\beta=(\beta_1,\cdots,\beta_l)\in \mathbb{R}^k$ is the unknown parameter and $Z_n(\beta)$ is the normalizing constant. In this paper, only sufficient statistics considered are  sub-graph counts, for e.g.  number of edges ($\sum_{i< j}x_{ij}$), number of two stars ($\sum_{i}\sum_{j<l,j,l\neq i}x_{ij}x_{il}$), number of triangles ($\sum_{i<j<l}x_{ij}x_{jl}x_{li}$), etc. 
\\



One of the main difficulties in estimation theory of these models is that the normalizing constant $Z_n(\beta)$ is not available in closed form. Explicit computation of the partition function takes time which is exponential in $n$, and so the calculation of MLE becomes infeasible.  One way out is to compute the MCMCMLE (see \cite{GT}), which approximates the  partition function by estimating it via Markov Chain Monte Carlo. Another way around is to compute the pseudo-likelihood estimator of Besag (\cite{B1},\cite{B2}), which depends only on the conditional distribution of one edge given the rest, which are easy to compute. However  theoretical properties of these estimators are  poorly understood in case of ERGM. Some recent progress  has been made in the theoretical properties on the ERGM models in the papers \cite{BSB} in (2008), and \cite{CD} in (2011), which is described below.
\\

\subsection{ Previous work}

Let $(G_i,1\leq i\leq l)$ denote a number of simple graphs, with $G_i=(V_i,E_i)$ denote the vertex set and the edge set respectively. Assume that $G_1$ is an just an edge, i.e. a graph with two vertices connected by an edge. The term in the exponent of an ERGM will also be referred to as the Hamiltonian, in analogy with Physics literature. 
\\

Consider  an ERGM with Hamiltonian of the form $n^2\sum\limits_{i=1}^l\beta_i\frac{N_{G_i}(x)}{n^{|V_i|}}$, where $N_{G_i}(x)$ denote the number of copies of  $G_i$'s in the graph $x$.  The parameter set   of interest is $\{\beta:\beta_i\geq 0,2\leq i\leq l\}\subset\R^l$, which will be referred to as the non-negative parameter domain. Note that the edge parameter need not be non negative in the non-negative domain. 
\\
For $0<p<1$, set \[\psi(p):=\sum\limits_{i=1}^l2\beta_i|E_i|p^{|E_i|-1},\quad \phi(p):=\frac{e^{\psi(p)}}{1+e^{\psi(p)}}.\] Depending on the parameter $\beta$, the equation $\phi(p)=p$ can have one or more solutions.  The main result of \cite{BSB} is the following:
\newline

\begin{itemize}
\item
If $\phi(p)=p$ has a unique root $p_0$ which satisfies $\phi'(p_0)<1$, then $\beta$ is said to be in the high temperature  regime. In this regime, the mixing time of Glauber dynamics is $O(n^2\log n)$, i.e. at most $Cn^2\log n$ for some $C<\infty$.

\item
If $\phi(p)$ has at least two roots both of which satisfy $\phi'(p)<1$, then $\beta$ is said to be in the low temperature regime. In this case, the mixing time of Glauber dynamics takes $e^{\Omega( n)}$ time (at least $ e^{Cn}$ for some $C>0$) to mix.  Further, this holds for any local Markov chain. 
\end{itemize}

\begin{remark}
This means  in particular that for MCMCMLE with a local Markov chain such as Glaubler dynamics, the mixing time of the Markov chain can be very large for some parameter values. Note however that the sampling method described in this paper (see Theorem \ref{global})  uses a non local chain, and so is not covered by the result in \cite[Theorem 6]{BSB}.
\\

As a comment, note that there are parameter values  $\beta$ which are not covered by the two cases of \cite{BSB}. These are referred to as critical points. 
\end{remark}

The main result of  \cite{CD} computes the limiting log partition function for all values of the parameter $\beta$, even outside non-negative domain. The limit obtained is in the form of an optimization problem which might be  intractable in general. In the non-negative domain, the limit  can be expressed in terms of the following $1$-$d$ optimization problem: \[\lim\limits_{n\rightarrow\infty}\frac{1}{n^2}\log Z_n(\beta)=\sup_{0<p<1}\{\sum\limits_{i=1}^k\beta_ip^{|E_i|}-I(p)\},\quad I(p)=\frac{1}{2}p\log p+\frac{1}{2}(1-p)\log(1-p).\]  It is easy to check that maximizing $p$ satisfies the same equation $p=\phi(p)$ as above. 
\\

\cite{CD} also proves  that in the non-negative regime, the model  generates data which are either very close (in the sense of cut-metric) to  an Erdos Renyi, or a mixture of Erdos Renyi. For a discussion on the cut metric, see \cite{CD} and the references therein. Since an Erdos Renyi graph is characterized by one and one parameter only, this seems to suggest that the model might be un-identifiable in the limit, and so if $k\geq 2$ the parameter $\beta$ might not be consistently estimable.  
\\

\subsection{ The two star model defined}

This paper considers a specific ERGM $\mathbb{P}_{n,\beta}$ given by the  Hamiltonian  \[\frac{\beta_2}{n-1}T(x)+\Big({\beta_1}+\frac{\beta_2}{n-1}\Big)E(x),\] where $T(x)=\sum_i\sum_{j<l,j,l\neq i}x_{ij}x_{il}$ and $E(x)=\sum_{i<j}x_{ij}$ are the number of two stars and edges respectively. Mathematically the model is given by
$$\P_{n,\beta}(x)=\frac{1}{Z_n(\beta)}e^{\frac{\beta_2}{n-1}T(x)+\Big({\beta_1}+\frac{\beta_2}{n-1}\Big)E(x)}, \quad x\in \mathcal{X}_n.$$ This is the two star model studied  by Park and Newmann \cite{PN} with parameters $(\beta_1+\beta_2/(n-1),\beta_2)$. The choice of the scaling is done to simplify  computations   later. In this paper the focus  is on non negative domain $\{(\beta_1,\beta_2):\beta_2>0\}$. Note that $\beta_2=0$ corresponds to an Erdos Renyi model with $p_n=\frac{e^{\beta_1+\beta_2/(n-1)}}{1+e^{\beta_1+\beta_2/(n-1)}}$.   Thus the two star model can be thought of as the simplest generalization of Erdos-Renyi model.
\\

\subsection{Outline}

In section 2  auxiliary variables are introduced to transform the discrete problem into a continuous one. As a by product, one obtains a sampling algorithm for the two star model.
 Section 3 uses heuristic analysis of the continuous model  to identify the phase transition boundary for the problem. 
 \\
 
 Section 4 contains  rigorous results confirming the identification of section 3. In particular, setting $(d_1,\cdots,d_n)$ denoting the labelled degrees of the graph $x$, it shows that there is a regime of parameters $\beta$ where $$\P_{n,\beta}\Big(\max_{i=1}^n\Big|\frac{d_i}{n-1}-p_0\Big|>\delta\Big)\le e^{-C(\delta)n}$$ for some $p_0\in (0,1)$,
 i.e. all the scaled degrees are close to one common value with very high probability. This regime will be referred to as the uniqueness regime. On the other hand, there is another regime 
 where all the scaled degrees converge  converge to either one of two points, i.e. there exists two points $0< p_1<p_2< 1$ such that $$\Big|\P_{n,\beta}\Big(\max_{i=1}^n\Big|\frac{d_i}{n-1}-p_j\Big|>\delta\Big)-\frac{1}{2}\Big|\le e^{-C(\delta)n}$$ for $j=1,2$. This regime will be referred to as the non uniqueness regime. This change illustrates the phase transition phenomenon in the two star model.
 
  Further the above decomposition covers the entire non negative domain $\{(\theta_1,\theta_2):\theta_2>0\}$ barring a single point $\theta_1=0,\theta_2=1/2$. This point will be referred to as the critical point.
 
 The proofs of Theorem \ref{ab} and Theorem \ref{ba} require a lot of technical estimates, but the results are easy to justify from a heuristic sense using the phase boundaries of section 3.
 \\
 
Section 5  contains histograms of the scaled degree sequence which validates  the results of  section  4. The simulations of section 5 is based on the algorithm of section 2.
\section{Simplifying the model}

\subsection{ Connection with the Ising model on the Line graph of the complete graph}
\

It so  happens that computations with $\{-1,1\} $ is much easier than with $\{0,1\}$, and so the symmetric $\{-1,1\}$ valued matrix $y=((y_{ij}))_{n\times n}$ is introduced as follows  $$y_{ij}:=(2x_{ij}-1)\in \{ -1,1\}, i\neq j,\quad y_{ii}:=0.$$ Note that the hamiltonian up to constants is given by \[\frac{\beta_2}{4(n-1)}T(y)+(\beta_1+\beta_2) E(y)=\frac{\theta_2}{n-1} T(y)+\theta_1 E(y),\]  where  $$\theta_2:=\frac{1}{4}\beta_2,\quad \theta_1:=\frac{1}{2}(\beta_1+\beta_2)$$ is a reparametrization, and $T(y),E(y)$ are given by the same formula with $x$ replaced by $y$, i.e. 
$$T(y)=\sum_i\sum_{j<l,j,l\neq i}y_{ij}y_{ij},\quad E(y)=\sum_{i<j}y_{ij}.$$ Thus $\P_{n,\beta}$ induces a probability on  $\{-1,1\}^{n(n-1)/2}$ which  is an  Ising model of statistical physics. The underlying graph of the Ising model is the graph $\mathbb{L}_n$ with edge set $\mathcal{E}:=\{(i,j):1\leq i<j\leq n\}$ as its vertex set, where two distinct vertices $e=(i,j)$ and $f=(k,l)$ are connected iff $\{(i,j)\cap (k,l)\}\neq \phi$, i.e. $i=k$ or $i=l$ or $j=k$ or $j=l$. It is easy to see that $\mathbb{L}_n$ is isomorphic to the line graph of the complete graph $K_n$. Also this Ising model is Ferro magnetic in terms of statistical physics terminology, as $\theta_2>0$.
\\



\subsection{Introducing the auxiliary variables and changing to a multivariate density on $\mathbb{R}^n$ }

This subsection introduces auxiliary variables $\phi=(\phi_1,\cdots,\phi_n)$ which gives a nice representation of the probability $\mathbb{P}_{n,\beta}$, along with a non local Markov chain which can be used to simulate  from the model.  The idea is motivated from a paper on the edge two star model by  J.Park  and M.E.J.Newmann \cite{PN}. 
 \\
 
  Setting $k_i:=\sum\limits_{j\neq i}y_{ij}$ it is easy to check that  
  \[T(y):=\frac{1}{2}[\sum\limits_ik_i^2-n(n-1)],\quad E(y):=\frac{1}{2}\sum\limits_{i=1}^nk_i,\] and so ignoring constants,  \[ \mathbb{P}_{n,\beta}(y)\propto e^{\frac{\theta_2}{2(n-1)}\sum\limits_{i=1}^n k_i^2+\frac{\theta_1}{2}\sum\limits_{i=1}^nk_i}.\]
Consider auxiliary random variables $\phi_i,1\leq i\leq n$ introduced by the following definition: 

Given $y$, let $\{\phi_i\}_{i=1}^n$ be  mutually independent and \[\phi_i{\sim} N\Big(\frac{k_i}{n-1},\frac{1}{(n-1)\theta_2}\Big).\] Thus the conditional density of $(\phi|y)$ is proportional to 
\[e^{-\frac{(n-1)\theta_2}{2}\sum\limits_{i=1}^n(\phi_i-\frac{k_i}{n-1})^2}, \]
and so  the joint likelihood of $(y,\phi)$ is proportional to \[e^{-\frac{(n-1)\theta_2}{2}\sum\limits_{i=1}^n\phi_i^2+\sum\limits_{i=1}^n(\theta_2\phi_i+\frac{\theta_1}{2})k_i}=e^{-\frac{(n-1)\theta_2}{2}\sum\limits_{i=1}^n\phi_i^2+\sum\limits_{i<j}[\theta_2(\phi_i+\phi_j)+\theta_1]y_{ij}}.\] This implies that conditional on $\phi$, the $y_{ij}$'s are mutually independent, and have the distribution \[\mathbb{P}_{n,\beta}(y_{ij}=1|\phi)=\frac{e^{\theta_2(\phi_i+\phi_j)+\theta_1}}{e^{\theta_2(\phi_i+\phi_j)+\theta_1}+e^{-\theta_2(\phi_i+\phi_j)-\theta_1}},\] where by a   slight abuse of notation,   $\mathbb{P}_{n,\beta}$ also denotes the joint law of $(y,\phi)$.
 \begin{remark}
The above construction is equivalent to the following representation:
\begin{align}\label{domm2}
\phi_i:=\frac{k_i}{n-1}+\frac{Z_i}{\sqrt{(n-1)\theta_2}}, Z_i\stackrel{i.i.d.}{\sim}N(0,1) \text{ independent  of }y
\end{align}
 \end{remark}
 \begin{defn}\label{d0}
Denote the marginal distribution  of $\phi$ under $\mathbb{P}_{n,\beta}$ by $\mathbb{F}_n$. $\mathbb{F}_n$ is absolutely  continuous with respect to Lebesgue measure on $\R^n$ and has the following  un-normalized density \[f_n(\phi):=e^{-\frac{(n-1)\theta_2}{2}\sum\limits_{i=1}^n\phi_i^2+\sum\limits_{i<j}\log\cosh[\theta_2(\phi_i+\phi_j)+\theta_1]}=e^{-\sum\limits_{i<j}p(\phi_i,\phi_j)},\quad p(x,y):=\frac{\theta_2}{2}(x^2+y^2)-\log\cosh(\theta_2(x+y)+\theta_1).\]
\end{defn}
Thus one can infer results about the distribution of $\phi$ from $f_n(.)$, and transfer them to conclusions about $y$ via (\ref{domm2}). This program will be carried out in section 4  to rigorously study concentration of degrees in the two star model.
\\

The above representation is summarized in the following theorem.
\begin{thm}\label{global}
\begin{enumerate}
\item[(a)] 
The law of $x$ under $\mathbb{P}_{n,\beta}$ has the following mixture representation:

Let $\phi\sim \mathbb{F}_n$, and given $\phi$, let $\{x_{ij}\}_{1\leq i<j\leq n}\in \{0, 1\}$ be mutually independent Bernoulli variables with parameter \[\frac{e^{\theta_2(\phi_i+\phi_j)+\theta_1}}{e^{\theta_2(\phi_i+\phi_j)+\theta_1}+e^{-\theta_2(\phi_i+\phi_j)-\theta_1}}=\frac{e^{2\theta_2(\phi_i+\phi_j)+2\theta_1}}{e^{2\theta_2(\phi_i+\phi_j)+2\theta_1}+1}\] Then $x$ has the same law as under $\mathbb{P}_{n,\beta}$. The conditional model $(x|\phi)$ is given by 
\[\mathbb{P}_{n,\beta}(x|\phi)=\frac{e^{\sum\limits_{i=1}^nd_i(x)(2\theta_2\phi_i+\theta_1)}}{\prod\limits_{1\leq i<j\leq n}\{e^{2\theta_2(\phi_i+\phi_j)+2\theta_1}+1\}}.\] 

\item[(b)]
Consider the following Gibb's Sampler:
\begin{itemize}
\item
Given $y$, let $\{\phi_i\}_{i=1}^n$ be mutually independent, with \[\phi_i\sim N\Big(\frac{k_i}{n-1},\frac{1}{\theta_2(n-1)}\Big),\quad k_i=\sum\limits_{j\neq i}y_{ij}.\]

\item
Given $\phi$, let $\{y_{ij}\}_{1\leq i<j\leq n}$ be mutually independent and taking values in $\{-1,1\}$ with \[\mathbb{P}(y_{ij}=1|\phi)=\frac{e^{\theta_2(\phi_i+\phi_j)+\theta_1}}{e^{\theta_2(\phi_i+\phi_j)+\theta_1}+e^{-\theta_2(\phi_i+\phi_j)-\theta_1}}.\]
\end{itemize}
\end{enumerate}
 Then  the distribution of $y$ after $l$ iterations of the Gibbs sampler  converge to the law of $y$ under $\mathbb{P}_{n,\beta}$ in total variation as $l\rightarrow\infty$.
\end{thm}

\begin{proof}

The proof of (a) follows from noting that $x_{ij}\leftrightarrow y_{ij}$ is a $1-1$ map. For the proof of (b), note that the Markov chain is irreducible aperiodic positive recurrent  with  $\mathbb{P}_{n,\beta}$ as its stationary distribution.
\end{proof}
\begin{remark}
The conditional distribution of $(x|\phi)$ is  also an exponential probability measure on the space of all simple graphs,  with the degree distribution as its sufficient statistics.This model is known as the $\beta$-model in statistical literature, and has been studied  in  \cite{CDS}, \cite{PN2}, and \cite{BD} among others.
\\
Part (b) of Theorem \ref{global} gives a way to simulate from the model $\mathbb{P}_{n,\beta}$.  The rates of convergence of this Markov chain has not been analyzed in this paper.
\end{remark}

\section{Identification of Phase transition boundary}

This section minimizes the  function $p(.,.)$ of definition \ref{d0} in a heuristic attempt to identify the phase transition boundary for the two star model.  The phase transition boundary for this model turns out to be the same as that of the Curie Weiss model of statistical physics.
\\

From the form of density of $\phi$ under $\mathbb{F}_n$ it follows that $\phi$ has more mass in areas where $p(.,.)$ is small, and so it makes sense to minimize $p(.,.)$ to identify the steady states of the model. Note that 
\begin{align}\label{domm}
p(x,y)=q(x+y)+\frac{\theta_2}{4}(x-y)^2,\quad q(t):=\frac{\theta_2}{4}t^2-\log \cosh(\theta_2t+\theta_1).
\end{align} The next lemma shows that the points of minima of $q$ determine the points of minima of $p(.,.)$.
\begin{lem}\label{min}

Let $U$ be one of the three open intervals $\{(0,\infty),(-\infty,0),\R\}$, and suppose there exists  a unique $\phi_0\in U$  such that  $q(t)$ has a global minima on $U$ at $t=2\phi_0$ with $q''(2\phi_0)>0$. Then there exists positive constants $\lambda_1>\lambda_2$ (depending on $\theta_1,\theta_2$) such that for all $x,y\in U$,

\begin{align}
&\label{e4}p(\phi_0,\phi_0)+\frac{\lambda_2}{2}[(x-\phi_0)^2+(y-\phi_0)^2]\leq p(x,y)\leq p(\phi_0,\phi_0)+\frac{\lambda_1}{2}[(x-\phi_0)^2+(y-\phi_0)^2].
\end{align}

\end{lem}

\begin{proof}
Define a function $r(t)$ on $U$ by \[r(t):=\frac{2(q(t)-q(2\phi_0))}{(t-2\phi_0)^2},t\neq 2\phi_0;\quad  r(2\phi_0)=q''(2\phi_0).\] Then by definition $r$ is continuous on $\mathbb{R}$. Since $\lim\limits_{|t|\rightarrow\infty}r(t)=\frac{\theta_2}{2}>0$, and  $r(t)>0$ for all $t$, it follows that  \[\lambda_2':=\inf_{t\in U} r(t)>0,\] which readily gives \[q(t)\ge q(2\phi_0)+\frac{\lambda_2'}{2}(t-2\phi_0)^2.\]Using (\ref{domm}) with $t=x+y$ this gives    
\begin{align*}
p(x,y)\geq p(\phi_0,\phi_0)+\frac{\theta_2}{4}(x-y)^2+\frac{\lambda_2'}{2}(x+y-2\phi_0)^2.
\end{align*}
Thus setting $\lambda_2=\min(\lambda_2',\frac{\theta_2}{2})$ gives  \[p(x,y)\ge p(\phi_0,\phi_0)+\frac{\lambda_2}{2}[(x-\phi_0)^2+(y-\phi_0)^2].\] Existence of $\lambda_1$ follows by a similar argument.
\end{proof}

\begin{remark}
Lemma \ref{min} readily gives that $p(.,.)$ has a global minima at $(\phi_0,\phi_0)$ on $U\times U$. The stronger conclusion of existence of $\lambda_1,\lambda_2$ will be used in section 4 to deduce some properties of the distribution $\mathbb{F}_n$.
\end{remark}
Lemma \ref{min} thus reduces the problem of minimization of $p(.,.)$ over $\R^2$ to a problem of minimization of $q(.)$ over $\R$. The later task is now carried out via sub-cases.


\begin{itemize}
\item{ $\Theta_{11}:\{\theta_1=0,\theta_2<1/2\}$}
\\

In this case $q''(t)>0$ and so $q(.)$ is strictly convex with a unique global minima at $0$. Thus the conditions of Lemma \ref{min} is satisfied with $\phi_0=0, U=\R$.
\\

\item{$\Theta_2:\{\theta_1=0,\theta_2>1/2\}$}
\\

Since $q(t)$ goes to $\infty$ as $|t|\rightarrow\infty$ the global minima is attained at a finite point.
Differentiating $q$ gives $q'(t)=\frac{\theta_2}{2}[t-2\tanh(\theta_2 t)]$ which has  exactly three real roots $0,\pm 2m$ where $m$ is a positive root of $t=\tanh(2\theta_2t)$. Also note that $q''(0)<0$, whereas $q''(\pm 2m)>0$. By symmetry it follows that  $\pm 2m$ are global minima of $q(.)$, and $0$ is a local maxima. Thus the conditions of Lemma \ref{min} is satisfied with either $\phi_0=m,U=(0,\infty)$ or $\phi_0=-m,U=(-\infty,0)$.
\\

\item{ $\Theta_3:\{\theta_1=0,\theta_2=1/2\}$}
\\

In this case $q''(t)\ge 0$ with equality at $t=0$ and so the function $q$ is convex but not strictly convex. In this case $q(.)$  has a unique global minima at $0$. However the conditions of Lemma \ref{min} is not satisfied as $q''(0)=0$. \\

\item{$\Theta_{12}:\{\theta_1>0,\theta_2>0\}$}
\\

In this case $q(-t)>q(t)$, and so it suffices to minimize over $t>0$. Also since $q(t)$ goes to  $\infty$ as $t\rightarrow\infty$, the global minima is not attained at a finite point $2m\ge 0$, say. Then $m$ must satisfy $q'(2m)=0$, which simplifies to
 $m=\tanh[2\theta_2m+\theta_1]$. But the last equation  has a unique strictly positive solution on $[0,\infty)$, and so  the global minima for $q$ is at $2m$ with $m>0$, where $m>0$ is a root of $t=\tanh(2\theta_2 t+\theta_1)$. Also $q''(2m)>0$, and so
 the conditions of Lemma \ref{min} hold with $\phi_0=m, U=\R$, where $m$ is the unique positive root of $t=\tanh(2\theta_2t+\theta_1)$.
 \\
 
  \item{$\Theta_{13}:\{\theta_1<0,\theta_2>0\}$}
  \\
  
   By symmetry, 
  the conditions of Lemma \ref{min} hold with $\phi_0=m,U=\R$, where now $m$ is the unique negative root of $t=\tanh(2\theta_2t+\theta_1)$.
\\

\end{itemize}
\begin{remark}
The domain $\Theta_1:=\Theta_{11}\cup \Theta_{12}\cup\Theta_{13}$ is  the uniqueness domain, as the global minimization of $q(.)$ occurs at a unique point. This is also known as the high temperature regime in statistical physics.
\\

The domain $\Theta_2$ is the non-uniqueness domain, as the minima is attained at two distinct points. This domain is known as the low temperature regime in statistical physics.
\\

The domain $\Theta_3$ is the critical point parameter configuration, as the function $q(.)$ changes its behavior at this point. 
\end{remark}
\begin{remark}
The assertions about the roots of the equation $t=\tanh(2\theta_2t+\theta_1)$ can be checked directly, or can be verified  from (\cite[Page 9]{DM}). It also follows from \cite{DM} that the phase transition boundary of the Curie-Weiss model is the same (the transition for Curie Weiss model is at $\theta_2=1$ instead of $1/2$, but this is due to the scaling chosen for the two star model). 
\end{remark}

\section{Statement and proofs of main results}
In order to state the results of this section, the following definition is introduced.
\begin{defn}
Let $\{a_n\}_{n\ge 1},\{b_n\}_{n\ge 1}$ be two sequences of positive real numbers. The notation $a_n=\Omega(b_n)$ means there exists a constant $C>0$ free of $n$ such that $a_n\ge C b_n$.
\end{defn}
The main results of this section are the two following theorems:
\begin{thm}\label{ab}
If $\theta\in \Theta_1$ then there exists unique $p_0\in (0,1)$ for which 
$$\P_{n,\beta}\Big(\max_{1\le i\le n}\Big|\frac{d_i}{n-1}-p_0\Big|>\delta\Big)\le e^{-\Omega(n)}.$$
\end{thm}
\begin{thm}\label{ba}
If $\theta\in \Theta_2$ then there exists distinct $p_1,p_2\in (0,1)$ for which 
$$\P_{n,\beta}\Big(\max_{1\le i\le n}\Big|\frac{d_i}{n-1}-p_j\Big|>\delta\Big)-\frac{1}{2}|\le e^{-\Omega(n)}.$$
\end{thm}
The two theorems will be proved via a series of lemmas.
\\
The first lemma provides a basic estimate which tells us that all the $\phi_i$'s are within a sub interval of $(-1,1)$  with high probability.
\begin{lem}\label{bound}

There exists $-1<m_1<m_2<1$ such that \begin{align}\label{e1}\mathbb{P}_{n,\beta}(\phi_i\notin [m_1,m_2]\text{ for some }i,1\le i\le n)\leq e^{-\Omega(n)}.\end{align}

\end{lem}

\begin{proof}

For a simple graph $x\in \mathcal{X}_n$ and a given edge $e\in \mathcal{E}$ define the simple graphs $x^{+e},x^{-e}\in \mathcal{X}_n$ as follows:
\begin{align*}
x^{+e}_f=&x^{-e}_f=x_f\text{ if }f\neq e,\\
x^{+e}_e=&1,\\
x^{-e}_e=&0,
\end{align*} i.e. $x^{+e}$ and $x^{-e}$ are basically the graph $x$ with the edge $e$ present or absent respectively, irrespective of whether it was present or absent in $x$ to begin with. 
\\

Setting \[a_1:=\frac{e^{\beta_1}}{1+e^{\beta_1}},a_2:=\frac{e^{2\beta_2+\beta_1}}{1+e^{2\beta_2+\beta_1}},\] 
note that $0<a_1<a_2<1$. Also, 
 for any $x\in \mathcal{X}_n$ 
 \begin{align*}
\log \frac{\p(x^{+e})}{\p(x^{-e})}=\frac{\beta_2}{n-1}\sum\limits_{f\in N(e)}x_f+\Big(\beta_1+\frac{\beta_2}{n-1}\Big)\leq \frac{2\beta_2(n-2)}{n-1}+\beta_1+\frac{\beta_2}{n-1}\leq \beta_1+2\beta_2=\log \frac{a_2}{1-a_2}.
\end{align*}
It follows by an application of  \cite[Theorem 2.3(c)]{Gr} that $\mathbb{P}_{n,\beta}|_x\leq \mathbb{Q}_{n,a_2}$ in the sense of stochastic ordering on graphs, where $\mathbb{Q}_{n,a}$ is an Erdos Renyi distribution on $\mathcal{X}_n$ with parameter $a$. A similar argument gives that  $\mathbb{P}_{n,\beta}|_x\geq \mathbb{Q}_{n,a_1}$, and so  for any  $\delta>0$, \[\mathbb{P}_{n,\beta}\Big( \frac{d_i(x)}{n-1}\notin[a_1-\delta,a_2+\delta]\text{ for some }i,1\le i\le n\Big)\leq e^{-\Omega(n)}.\]  Also recall that $\frac{k_i(y)+n-1}{2}=d_i(x)$, and so \[\mathbb{P}_{n,\beta}\Big(\frac{k_i}{n-1}\notin [2(p_1-\delta)-1,2(p_2+\delta)-1]\text{ for some }i,1\le i\le n\Big)\leq e^{-\Omega(n)}.\] The conclusion follows on using   (\ref{domm2}) 
and noting that  \[\mathbb{P}\Big(\frac{|Z_i|}{\sqrt{(n-1)\theta_2}}>\delta\text{ for some i},1\le i\le n\Big)\leq e^{-\Omega(n)},\] for any $\delta>0$, with $Z_i\stackrel{i.i.d.}{\sim} N(0,1)$.

\end{proof}
The second Lemma builds on Lemma \ref{bound} to develop concentration results for all the $\phi_i$'s simultaneously. These estimates will be used for the proof of Theorems \ref{ab} and \ref{ba}.
\begin{lem}\label{exp}
Suppose the conditions of Lemma \ref{min} hold.
Then
\begin{align}
&\label{e3}\mathbb{P}_{n,\beta}(\sum\limits_{i=1}^n(\phi_i-\phi_0)^2> M|\phi\in U^n)\leq e^{-\Omega(n)},\\
&\label{e4}\P_{n,\beta}(\max_{1\le i\le n}|\phi_i-\phi_0|>\delta|\phi\in U^n)\le e^{-\Omega(n)}.
\end{align}
\end{lem}
\begin{proof}
Denote by $\P_{n,\beta,U}$ the probability $\P_{n,\beta}$ conditioned on the event $\phi\in U^n$.
For $\phi\in U^n$, an application of Lemma \ref{min} gives 
  \[\frac{(n-1)\lambda_2}{2}\sum\limits_{i=1}^n(\phi_i-\phi_0)^2\leq -\log f_n(\phi)  \leq \frac{(n-1)\lambda_1}{2}\sum\limits_{i=1}^n(\phi_i-\phi_0)^2,\]  where $f_n(.)$ is the unnormalized density corresponding to $\mathbb{F}_n$ (see definition (\ref{d0}). Thus for any   $M>0$,  
\begin{align*}
\mathbb{P}_{n,\beta,U}(\sum\limits_{i=1}^n(\phi_i-\phi_0)^2>M)
\leq\frac{\int\limits_{\{\sum\limits_{i=1}^n(\phi_i-\phi_0)^2>M\}}e^{-\frac{(n-1)\lambda_2}{2}\sum\limits_{i=1}^n(\phi_i-\phi_0)^2}d\phi}{\int\limits_{U^n}e^{-\frac{(n-1)\lambda_1}{2}\sum\limits_{i=1}^n(\phi_i-\phi_0)^2}d\phi}
\leq \Big(\frac{\lambda_1}{\lambda_2}\Big)^{\frac{n}{2}}\frac{\mathbb{P}(\chi_n^2\geq (n-1)\lambda_2 M)}{\mathbb{P}(\phi_0+\frac{Z}{\sqrt{(n-1)\lambda_1}}\in U)^n},
\end{align*} 
where $Z\sim N(0,1)$ and $\chi_n^2$ is a chi-square random variable with $n$ degrees of freedom. Also, the denominator converges to $1$ as $\phi_0\in  U$.
\\

Proceeding to bound the numerator, first note that by Markov's inequality,
\[\log\mathbb{P}(\chi_n^2\geq t)\leq  -\frac{t}{2}+\frac{n}{2}-\frac{n}{2}\log(\frac{n}{t}).\] Plugging in $t=(n-1)\lambda_2 M$ and letting $n\rightarrow\infty$ we conclude \[\limsup\limits_{n\rightarrow\infty}\frac{1}{n}\log \mathbb{P}_{n,\beta,U}(\sum\limits_{i=1}^n(\phi_i-\phi_0)^2>M)\leq \log\Big(\frac{\lambda_1}{\lambda_2}\Big)-\frac{\lambda_2 M-1}{2}+\frac{1}{2}\log(\lambda_2 M).\] Since the r.h.s. of the last inequality goes to $-\infty$ as $M\rightarrow\infty$, there exists $M<\infty$  such that the r.h.s. is negative, from which (\ref{e3}) follows. 
\\\\

Proceeding to prove (\ref{e4}),  an application of (\ref{e3}) gives
note that  
\begin{align*}
\P_{n,\beta,U}(|\phi_1-\phi_0|>\delta)=\P_{n,\beta,U}(|\phi_1-\phi_0|>\delta,\sum\limits_{i=2}^n(\phi_i-\phi_0)^2\le M)+e^{-\Omega(n)},
\end{align*}
where the first term in the right hand side can be written as
 \begin{align}\label{enew}
 \P_{n,\beta,U}(|\phi_1-\phi_0|>\delta,\sum\limits_{i=2}^n(\phi_i-\phi_0)^2\le M)
= \E_{\P_{n,\beta,U}}\Big[\P_{n,\beta,U}(|\phi_1-\phi_0|>\delta|\phi_2,\dots,\phi_n)1\{\sum\limits_{i=2}^n(\phi_i-\phi_0)^2\le M\}\Big].
\end{align}
The conditional density of $(\phi_1|\phi_i,i\geq 2)$ is proportional to $\prod\limits_{i=2}^ne^{-p(\phi_1,\phi_i)}$ with
\[e^{-(n-1)p(\phi_0,\phi_0)-\frac{(n-1)\lambda_1}{2}(\phi_1-\phi_0)^2-\frac{\lambda_1}{2}\sum\limits_{i=2}^n(\phi_i-\phi_0)^2}\leq  \prod\limits_{i=2}^ne^{-p(\phi_1,\phi_i)}\leq e^{-(n-1)p(\phi_0,\phi_0)-\frac{(n-1)\lambda_2}{2}(\phi_1-\phi_0)^2-\frac{\lambda_2}{2}\sum\limits_{i=2}^n(\phi_i-\phi_0)^2}\]
by Lemma \ref{min}, and so on the set $\{\sum\limits_{i=2}^n(\phi_i-\phi_0)^2\le M\}$,
\begin{align*}
\P_{n,\beta,U}(|\phi_1-\phi_0|>\delta|\phi_2,\cdots,\phi_n)
  \leq &e^{\frac{\lambda_1-\lambda_2}{2}\sum\limits_{i=2}^n(\phi_i-\phi_0)^2}\sqrt{\Big(\frac{\lambda_1}{\lambda_2}\Big)}\frac{\mathbb{P}(|Z|>\delta\sqrt{(n-1)\lambda_2})}{\mathbb{P}(\phi_0+\frac{|Z|}{\sqrt{(n-1)\lambda_1}}\in U)}\\
  \leq&e^{\frac{(\lambda_1-\lambda_2)M}{2}}\sqrt{\Big(\frac{\lambda_1}{\lambda_2}\Big)}\frac{\mathbb{P}(|Z|>\delta\sqrt{(n-1)\lambda_2})}{\mathbb{P}(\phi_0+\frac{|Z|}{\sqrt{(n-1)\lambda_1}}\in U)},
 \end{align*}
 where  $Z\sim N(0,1)$. 
Finally note that $\mathbb{P}(\frac{|Z|}{\sqrt{(n-1)\lambda_1}}\in U)$ converges to $1$ as before, and $\P(|Z|>\delta\sqrt{(n-1)\lambda_2})=e^{-\Omega(n)}$. Plugging these estimates back in (\ref{enew}) completes  the proof of (\ref{e4}).
\\

\end{proof}
Having proven the crucial Lemma \ref{exp}, the proof of Theorem \ref{ab} is immediate.
\begin{proof}[Proof of Theorem \ref{ab}]
By the calculations of section 3, for $\theta\in \Theta_1$ there exists a unique $\phi_0$ satisfying the conditions of Lemma \ref{min} with $U=\R$. Thus an application of (\ref{e4}) of Lemma \ref{exp} along with the representation (\ref{domm2}) gives the desired conclusion with $p_0=(\phi_0+1)/2$.
\end{proof}

The proof of Theorem \ref{ba} requires a further lemma which is specific to $\theta\in \Theta_2$.
\begin{lem}\label{lfin}
For $\theta\in\Theta_2,$ $$\Big|\mathbb{P}_{n,\beta}(\{\phi_i>0,1\leq i\leq n\})+\mathbb{P}_{n,\beta}(\{\phi_i<0,1\leq i\leq n\})-1\Big|\leq e^{-\Omega(n)}.$$
\end{lem}
The proof of Lemma \ref{lfin} uses a detailed analysis of the function $f_n(\phi)$, and has been moved to the appendix.
\\

The proof of Theorem \ref{ba} is again immediate from Lemma \ref{lfin}.
\begin{proof}[Proof of Theorem \ref{ba}]
 Since $\theta_1=0$, the density of $\phi$ is symmetric in the sense $f_n(\phi)=f_n(-\phi)$. This along with Lemma \ref{lfin} gives 
 $$\Big|\P_{n,\beta}(\phi_i>0,1\le i\le n)-\frac{1}{2}\Big|=e^{-\Omega(n)}.$$
 Since the conditions of Lemma \ref{min} hold with $\phi_0=m, U=(0,\infty)$, an application of Lemma \ref{bound} gives that
 $$\P_{n,\beta}(|\phi_i-\phi_0|>\delta\text{ for some }i,1\le i\le n|\phi\in (0,\infty)^n)\le e^{-\Omega(n)}.$$ Combining these two results give
 $$\Big|\P_{n,\beta}(|\phi_i-\phi_0|>\delta\text{ for some }i,1\le i\le n)-\frac{1}{2}\Big|\le e^{-\Omega(n)}.$$ This along with the representation (\ref{domm2}) gives 
 $$\Big|\P_{n,\beta}(|\frac{d_i}{n-1}-p_1|>\delta\text{ for some }i,1\le i\le n)-\frac{1}{2}\Big|\le e^{-\Omega(n)},$$ where $p_1=\frac{m+1}{2}$. A similar argument shows that $$\Big|\P_{n,\beta}(|\frac{d_i}{n-1}-p_2|>\delta\text{ for some }i,1\le i\le n)-\frac{1}{2}\Big|\le e^{-\Omega(n)},$$ where $p_2=\frac{1-m}{2}$, thus completing the proof of the theorem.
\end{proof}

\section{Simulations}

In all the simulations below the number of vertices $n$ has been taken to be $n=1000$, and the burn in period has been taken to be $500$. The plotted diagrams are the histograms of the scaled degree distributions, i.e. histograms of the vector $\{d_i/(n-1),1\le i\le n\}$.

\subsection{Domain $\Theta_{11}$}
The   parameters chosen for the first diagram are $\theta_1=0,\theta_2=.25$.

\begin{figure}[htbp]
\begin{center}
\includegraphics[height=5in,width=5in]{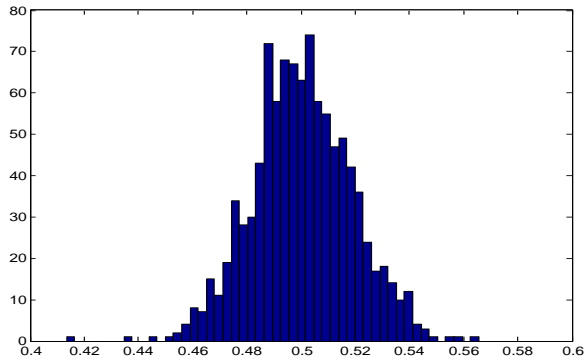}  
\caption{Histogram of degrees in domain $\Theta_{11}$}
\end{center}
\end{figure}
The histogram has a high mass near $0.5$. This agrees with Theorem \ref{ab}, which predicts that this domain all scaled degrees will converge to $p_0=.5$.  The maximum and minimum scaled degree are $0.5656$ and $0.4134$ respectively, and the average is $0.5005$.

\subsection{Domain $\Theta_{12}$}
The   parameters for the second figure are $\theta_1=.25,\theta_2=.25$.  

\begin{figure}[htbp]
\begin{center}
\includegraphics[height=5in,width=5in]{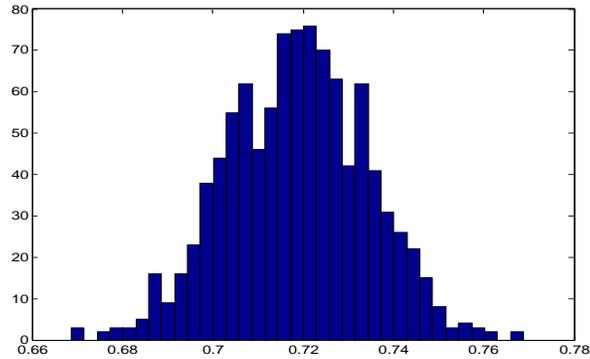}  
\caption{Histogram of degrees in domain $\Theta_{12}$}
\end{center}
\end{figure}
The histogram has a high mass near $0.72$.  The maximum and minimum scaled degree are $0.7688$ and $0.6687$ respectively, and the average is $0.7188$. In this domain the predicted limit of scaled degrees is as follows:

The limit is given by $p_0=(m+1)/2$, where $m$ is the unique positive root of $t=\tanh(2\theta_2+\theta_1)$. A plot of $t$ vs $\tanh(2\theta_2t+\theta_1)$ gives the approximate intersection point to be $m=0.4370$, which gives $p_0=0.7185$. Thus the theoretical predictions agree with the simulation results. 

\begin{figure}[htbp]
\begin{center}
\includegraphics[height=5in,width=5in]{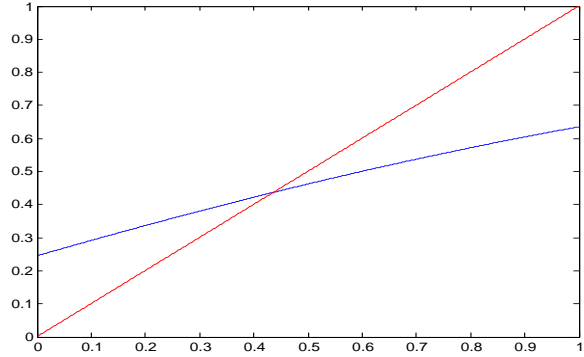}  
\caption{Plot of $t$ vs $\tanh(2\theta_2t+\theta_1)$ for $\theta_1=\theta_2=0.25$}
\end{center}
\end{figure}

\newpage
\subsection{Domain $\Theta_{2}$}

The third and fourth figures correspond to two independent simulations of the histogram of the scaled degree distribution from the model with parameter $\theta_1=0,\theta_2=0.55$. 
\\

\begin{figure}[htbp]
\begin{center}
\includegraphics[height=6in,width=6in]{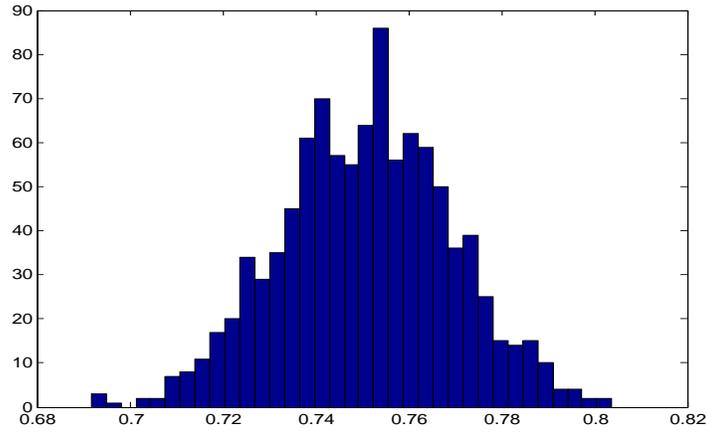}  
\caption{Histogram of degrees in domain $\Theta_2$}
\end{center}
\end{figure}

In the first  simulation the histogram  has a high mass near $0.75$. The maximum and minimum scaled degree are $0.8038$ and $0.6917$ respectively, and the average is $0.7510$. 
\newpage

\begin{figure}[htbp]
\begin{center}
\includegraphics[height=6in,width=6in]{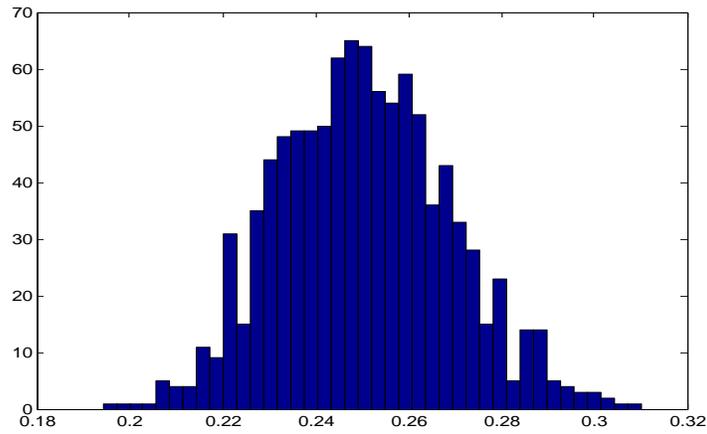}  
\caption{Histogram of degrees in domain $\Theta_2$}
\end{center}
\end{figure}

In the second simulation the histogram  has a high mass near $0.25$. The maximum and minimum scaled degree are $0.3103$ and $0.1942$ respectively, and the average is $0.2507$. 

\newpage

 Theorem \ref{ba} predicts this dual behavior, and further gives a way to compute  the two limits as follows:
 
 The limiting scaled degrees will converge to either $p_1=\frac{1+m}{2}$ or $p_2=\frac{1-m}{2}$, where $m$ is the unique positive root of the  equation $t=\tanh(2\theta_2t)$.

\begin{figure}[htbp]
\begin{center}
\includegraphics[height=6in,width=5in]{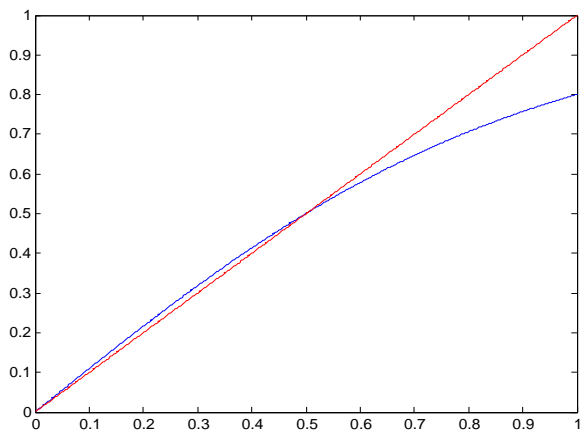}  
\caption{Plot of $t$ vs $\tanh(2\theta_2t)$ for $\theta_2=0.55$}
\end{center}
\end{figure}
From a simultaneous plot of $t=\tanh(2\theta_2t)$, the approximate point of intersection is $m=0.5020$, which gives $p_1=0.7510$ and $p_2=0.2490$, thus again agreeing with the simulations.
\newpage
\section{Conclusion}
The  phase transition of the edge two star model has been illustrated by  theoretical results as well as simulations. The different parameter domains corresponding to phase transition behavior has been explicitly characterized.  A simulating algorithm using auxiliary variables has been proposed for simulating from this model. 
\\

Unlike \cite{CD}, the calculations in this paper is very specific to the edge two star model, and does not generalize to other ERGMs such as the edge triangle model. It would be interesting to see if such concentration of degrees holds for such models.
\section{Acknowledgement}

I'm  grateful to my advisor Dr. Persi Diaconis for introducing me to this problem, and for his continued help and support during my Ph.D.

\section{Appendix}
The appendix carries out a proof of Lemma \ref{lfin}. 
Recall from section 3 that in this domain $p(x,y)$ has two global minima at $\pm(m,m)$. The first lemma shows  that most of  the $|\phi_i|$'s are close to  $m$ with high probability.  

\begin{lem}\label{l21}
If $\theta\in \Theta_2$,  there exists $M_1<\infty$ such that
     \begin{align*}
\mathbb{P}_{n,\beta}(\#\{i:|\phi_i|\notin I\}> M_1)\leq e^{-\Omega(n)},
\end{align*}
where $I=:(m/2,3m/2)$.
\end{lem}
  
\begin{proof}
  
Denoting the above set by $A_n$, it suffices to show that 
\begin{align}\label{fg3}
\frac{\int\limits_{A_n}f_n(\phi)d\phi}{\int\limits_{I^n}f_n(\phi)d\phi}\leq e^{-\Omega(n)},
\end{align} 
Proceeding to show $(\ref{fg3})$, note that $p(x,y)\geq p(|x|,|y|)$, and so for any quadrant $\mathcal{Q}$ (out of the $2^n$ possible),
\begin{align*}
\int\limits_{\mathcal{Q}\cap A_n}f_n(\phi)d\phi\leq \int\limits_{\mathcal{Q}_1\cap {A}_n}f_n(\phi)d\phi,
\end{align*}
where $\mathcal{Q}_1$ is the first quadrant in $\mathbb{R}^n$. 
Concentrating on $x,y>0$, using Lemma \ref{min} gives that there exists positive constants $\lambda_1>\lambda_2$ such that for all $x,y>0$ we have 
\[\frac{\lambda_2}{2}[(x-m)^2+(y-m)^2\leq p(x,y)\leq \frac{\lambda_1}{2}[(x-m)^2+(y-m)^2].\] Thus
\begin{align*}
\frac{\int\limits_{\mathcal{Q}_1\cap A_n}f_n(\phi)d\phi}{\int\limits_{I^n}f_n(\phi)d\phi}\leq \Big(\frac{\lambda_1}{\lambda_2}\Big)^{n/2}\frac{\mathbb{P}(\#\{i:\frac{Z_i}{\sqrt{(n-1)\lambda_2}|}>m/2\}>M_1 )}{\mathbb{P}(\frac{|Z|}{\sqrt{(n-1)\lambda_1}}<m/2)^n}
\end{align*}
where $Z,Z_i\stackrel{i.i.d.}{\sim}N(0,1)$. The probability in the denominator converges to $1$. By a union bound over the possible choice of indices, the probability in the numerator is bounded by $$\Big({{n}\atop{M_1}}\Big)e^{-M_1\Omega(n)}.$$ Summing up over all the $2^n$ quadrants gives
\[\mathbb{P}_{n,\beta}(\#\{i:|\phi_i|\notin I,1\leq i\leq n\}>M_1)\leq 2^n\Big(\frac{\lambda_1}{\lambda_2}\Big)^{n/2}\Big({{n}\atop{M_1}}\Big)e^{-M_1\Omega(n)}.\]
Choosing $M_1$ fixed but large enough gives (\ref{fg3}), and hence concludes the proof of the Lemma
\\
\end{proof}

Building on Lemma \ref{l21} the proof of Lemma \ref{lfin} is carried out next.

\begin{proof}[Proof of Lemma \ref{lfin}]
Letting \begin{align*}
I_1=&I_1(\phi):=\{i:\phi_i\in I\},\\
I_2=&I_2(\phi):=\{i:\phi_i\in -I\},
\end{align*}
the first claim is that there exists $M_2<\infty$ such that
\begin{align}\label{fg4}
\mathbb{P}_{n,\beta}(I_1>M_2,I_2>M_2)\leq e^{-\Omega(n)}.
\end{align} 
To show (\ref{fg4}) first note that   there exists $C_1>0$ such that  for all $x\in I,y\in -I$,
\begin{align*}
p(x,y)-p(|x|,|y|)\ge C_1.
\end{align*}
Indeed, this follows from the fact that $p(x,y)>p(|x|,|y|)$ for all $(x,y)$ on $I\times -I$  which is compact.  But  this readily gives 
$f_n(\phi)\leq f_n(|\phi|)e^{-I_1I_2C_1}$. Now
\begin{align}\label{fg5}
\mathbb{P}_{n,\beta}(I_1>M_2,I_2>M_2)\leq&\mathbb{P}_{n,\beta}(I_1>M_2,I_2>M_2,I_1+I_2\ge n-M_1)+\mathbb{P}_{n,\beta}(I_1+I_2<n-M_1),
\end{align}
with the second term  bounded by $e^{-\Omega(n)}$ by Lemma \ref{l21}. For the first term note that the events $$I_1>M_2,I_2>M_2,I_1+I_2\geq n-M_1$$
imply $ I_1I_2>M_2(n-M_1-M_2),$ and so 
\begin{align*}
\mathbb{P}_{n,\beta}(I_1>M_2,I_2>M_2,I_1+I_2\ge n-M_1)\leq \frac{2^n  e^{-M_2(n-M_1-M_2)C_1}\int\limits_{\mathcal{Q}_1}f_n(\phi)d\phi}{\int\limits_{\mathcal{Q}_1}f_n(\phi)d\phi}=2^ne^{-M_2(n-M_1-M_2)C_1}
\end{align*}
Thus choosing $M_2$ large enough enough gives (\ref{fg4}). Combining Lemma \ref{l21} and (\ref{fg4}) readily gives
\begin{align}\label{fg6}
\mathbb{P}_{n,\beta}(I_1<n-M_3,I_2<n-M_3)\leq e^{-\Omega(n)},
\end{align}
with $M_3:=M_1+M_2$, i.e. with high probability at least $n-M_3$ of the $\phi_i$'s are in exactly one of $\pm I$. 

To complete the proof of the lemma, setting $J_1:=\{i:\phi_i>0\},J_2:=\{i:\phi_i<0\}$, it suffices to show that
\[\mathbb{P}_{n,\beta}(J_1<n,J_2<n)\leq e^{-\Omega(n)}.\] To this effect, note that
$J_1\geq I_1,J_2\geq I_2$ and so  (\ref{fg6}) gives 
\begin{align*}
\mathbb{P}_{n,\beta}(J_1<n,J_2<n)\leq2\mathbb{P}_{n,\beta}(J_1<n,J_2<n,I_1\geq n-M_3)+\mathbb{P}_{n,\beta}(I_1<n-M_3,I_2<n-M_3)
\end{align*}
The second term is $e^{-\Omega(n)}$ by (\ref{fg5}). Turning to deal with the first term, note that there exists $C_2>0$ such that for  $x\in I,y\leq 0$,
\begin{align*}p(|x|,|y|)-p(x,y)\geq C_2.\end{align*}
 Indeed,  this function is positive point-wise on compact subsets of $I\times (-\infty,0)$, and their difference goes to $\infty$ if $y\rightarrow-\infty$. 
\\

Now the events $$J_1+J_2=n,\quad J_1\ge I_1>n-M_3,\quad J_2\geq  1$$ imply that there exists at least $(n-M_3)$ pairs $(i,j)$ such that $\phi_i\in I,\phi_j<0$. This readily gives  $f_n(\phi)\leq f_n(|\phi|)e^{-(n-M_3)C_2}$. Also $I_1\geq n-M_3$ can occur only on at most $2^{M_3}\Big({{n}\atop{M_3}}\Big)$ quadrants, and so by a union bound,
\begin{align*}
\mathbb{P}_{n,\beta}(J_1<n,J_2<n,I_1\geq n-M_3)\leq 2^{M_3}\Big({{n}\atop{M_3}}\Big)e^{-(n-M_3)d_2}\frac{\int\limits_{\mathcal{Q}_1}f_n(\phi)d\phi}{\int\limits_{\mathcal{Q}_1}f_n(\phi)d\phi}\leq 
(2n)^{M_3}e^{-(n-M_3)d_2}\leq e^{-\Omega(n)},
\end{align*}
completing the proof of the lemma.

\end{proof}

\end{document}